\documentclass[a4paper,12pt,reqno]{article}
\usepackage[T1]{fontenc}
\usepackage{authblk}
\usepackage{epsfig,graphicx,subfigure}
\usepackage{amsthm,amsmath,amssymb,amsfonts}
\usepackage{color,xcolor}
\usepackage[all]{xy}
\usepackage[top=40mm, bottom=40mm, left=40mm, right=35mm]{geometry}
\usepackage{tikz-cd}
\usepackage{hyperref}

\newtheorem{theorem}{Theorem}[section]

\newtheorem{proposition}[theorem]{Proposition}
\newtheorem{corollary}[theorem]{Corollary}
\theoremstyle{definition}

\numberwithin{equation}{section}

\begin{document}
\setcounter{page}{1}
\title{\vspace{-1.5cm}
\vspace{.5cm}
\vspace{.7cm}
{\large{\bf  On the Phase Retrievable Sequences}}}
 \date{}
\author{{\small \vspace{-2mm} 
 F. Javadi$^1$\footnote{Corresponding author}~ and M. J. Mehdipour$^{1}$}}
\affil{\small{\vspace{-4mm}  $^1$ Department of Mathematics, Shiraz University of Technology, P. O. Box $71555$-$313$, Shiraz,  Iran.}}
\affil{\small{\vspace{-4mm} f.javadi@sutech.ac.ir}}
\affil{\small{\vspace{-4mm} mehdipour@sutech.ac.ir}}
\maketitle
\hrule
\begin{abstract}
\noindent
 In this paper, we study phase retrievable sequences and give a characterization of phase retrievability of a sequence of bounded linear operators on a Hilbert space $H$; in particular, for $H=\ell_2^d(\Bbb{C})$. We also give several approaches for constructing
phase retrievable sequences.
Then, we investigate the property of phase retrieval for $g$-frames and frames.
 \end{abstract}

\\
\hrule
\vspace{0.5 cm}
\baselineskip=.55cm
\section{Introduction}
Throughout this paper, $H$ is a complex Hilbert space and
$\mathcal{L}(H)$ is  the space  of  all bounded linear operators on $H$.
A sequence $\{\Lambda_k\}_{k\in \Bbb{Z}}$
in $\mathcal{L}(H)$ is called \textit{phase retrieval} if for every
$x,y \in H$  satisfying
\begin{eqnarray*}
\|\Lambda_k(x)\|=\|\Lambda_k(y)\|
\end{eqnarray*}
 for all  $k \in \Bbb{Z}$, then
\begin{eqnarray*}
x= \alpha y  ~~~ \text{with} ~~~|\alpha |=1.
\end{eqnarray*}
 It is routine to check that  for every $n \in \Bbb{N}$, $\{\Lambda^n_k\}_{k \in \Bbb{Z}}$ is also  phase retrieval.
Also, 
if  $\{\Lambda_k\}_{k \in \Bbb{Z}}$ is a phase retrievable sequence of $\Lambda_k$  normal operators on $H$,
 then $\{\Lambda^*_k\}_{k \in \Bbb{Z}}$ is  phase retrieval. 
Note that if $\Lambda_k(x)=0$, then
\begin{eqnarray*}
\|\Lambda_k(x)\|=\|\Lambda_k(0)\|.
\end{eqnarray*}
So phase retrievability of
 $\{\Lambda_k\}_{k\in \Bbb{Z}}$ implies that $x=0$.
 Thus $\Lambda_k$ is injective for all $k\in \Bbb{Z}$.

A sequence $\{x_k\}_{k \in \Bbb{Z}}$ in  $H$ is called \textit{phase retrieval}, if for every  $x,y \in H$ and $k\in\Bbb{Z}$, from
$|\langle x, x_k \rangle |= |\langle y, x_k \rangle |$
implies that there exists $\alpha\in\Bbb{C}$ with $|\alpha|=1$ and
$x = \alpha y.$
We define the equivalence relation ``$\sim$'' on $H\times H$ as follows: 
$x\sim y$ if and only if there exists $\alpha\in\Bbb{C}$ with $|\alpha|=1$ such that $x=\alpha y$.
 Put $H_r:=H/\sim$. It is easy to see that $H_r=H/\Bbb{T}$, where $\Bbb{T}$ is the complex unit circle. Balan et al. [3]  introduced the function $\Bbb{M}: H_r\rightarrow \ell^2(\Bbb{I})$ which is defined by
\begin{eqnarray*}
\Bbb{M}(\hat{x})=\{|\langle x, x_k\rangle|\}_{k\in\Bbb{I}},
\end{eqnarray*}
where $\{x_k\}_{k \in \Bbb{I}}$ is a frame for $H$ and $\Bbb{I}$ is a finite or countable subset of $\Bbb{Z}$. They gave conditions under which $\Bbb{M}$ is injective. These results have applications in quantum mechanic; see [16]. Note that if $H$ is a real Hilbert space, then $H_r=H/\{\pm 1\}$. In the case where, $\hbox{dim}(H)=M$, Cahill et al. [7] considered the function $\Bbb{A}: \Bbb{R}^M/\{\pm 1\}\rightarrow \Bbb{R}^N$ defined by
\begin{eqnarray*}
\Bbb{A}(x)(n)=\|P_n(x)\|^2,
\end{eqnarray*}
where $P_n: \Bbb{R}^M\rightarrow W_n$ is the orthogonal projection on a subspace $W_n$ of $H$ for $n= 1, ..., N$. They characterized the subspaces $\{W_n\}_{n=1}^N$
for which $\{\|P_n(x)\|\}_{n=1}^N$
are injective for all $x\in H$. They also used the phrase ``$\{W_n\}_{n=1}^N$
allows phase retrieval'' instead of the phrase ``$\Bbb{A}$ is injective''.
Some authors continued these investigations [1, 2, 4, 5, 6, 12, 15]. For example, Han and Juste [11] studied phase-retrievable operator-valued frames and gave several characterizations for them. They found some relations between  phase-retrievable operator-valued frames, projective group representation frames and representations of quantum channels. These results naturally leads us to study injectivity of the function $\Bbb{A}:H_r/\Bbb{T}\rightarrow \ell^2(\Bbb{Z})$ defined by
\begin{eqnarray*}
\Bbb{A}(x)(k)=\|\Lambda_k(x)\|^2,
\end{eqnarray*}
where $\{\Lambda_k\}_{k\in\Bbb{Z}}$ is a sequence in $\mathcal{L}(H)$.

The outline of the paper is as follows.
In Section 2, we investigate phase retrievability of a sequence $\{\Lambda_k\}_{k\in\Bbb{Z}}$ in $\mathcal{L}(H)$ and provide some necessary and sufficient conditions under which $\{\Lambda_k\}_{k\in\Bbb{Z}}$ to be phase retrieval. We also show that $\{\Lambda_k\}_{k\in\Bbb{Z}}$ is phase retrieval 
(a  phase retrievable g-frame) if and only if $\{\Lambda_kT\}_{k\in\Bbb{Z}}$  is phase retrieval  (a  phase retrievable g-frame) when $T\in\mathcal {L}(H)$ is injective (surjective).
Section 3 is devoted to the study of a phase retrievable frame $\{x_k\}_{k\in\Bbb{Z}}$ in $H$. We establish that if $\{x_k\}_{k\in\Bbb{Z}}$ is a phase retrievable frame, then $\{Tx_k\}_{k\in\Bbb{Z}}$ is a phase retrievable frame, where $T\in\mathcal {L}(H)$ is injective. The converse remaind valid if $T$ is injective and self-adjoint. In particular, $\{x_k\}_{k\in\Bbb{Z}}$ is a phase retrievable frame if and only if the canonical dual of $\{x_k\}_{k\in\Bbb{Z}}$ is phase retrieval.
\section{Characterization of phase retrievable sequences}
We commence this section with the main result of this paper.
\begin{theorem}\label{j2-7}
A sequence $\{ \Lambda_k \}_{k \in \Bbb{Z}}$ in
 ${\cal L}(H)$ is not  phase retrieval  if and only if there exist nonzero elements $x, y\in H$
such that $x \not \in i \Bbb{R} y$ and
$\hbox{Re}\langle \Lambda_k(x), \Lambda_k(y) \rangle=0$
for all $ k \in \Bbb{Z}$.
\end{theorem}
\begin{proof}
First, assume that
$\{ \Lambda_k \}_{k \in \Bbb{Z}}$
 is not phase retrieval.
Then there exist $f, g \in H$ such that for every $\alpha\in\Bbb{C}$
with
 $|\alpha|=1$
 we have $f \not = \alpha g$ and for any $k\in\Bbb{Z}$
\begin{eqnarray}\label{f11}
\langle \Lambda_k(f), \Lambda_k(f) \rangle =\langle \Lambda_k(g), \Lambda_k(g) \rangle.
\end{eqnarray}
Let
$x=\frac{1}{2}(f+g)$
and
$y=\frac{1}{2}(f-g)$.
Then
$f=x+y$
and
$g=x-y$.
It follows from (\ref{f11})  that for every
$k\in\Bbb{Z}$
\begin{eqnarray*}
\hbox{Re}\langle \Lambda_k (x), \Lambda_k(y) \rangle =0.
\end{eqnarray*}
Clearly, $x$ and $y$ are nonzero.
If $x=i ay$ for some $a \in \Bbb{R}$, then
$f=(ia+1)y $ and $g=(ia-1)y.$
This shows that
\begin{eqnarray*}
\frac{f}{g}=\frac{ia+1}{ia-1} \,\,\,\,\,  \hbox{and} \,\,\,\,\,
\bigg| \frac{ia+1}{ia-1} \bigg|=1,
\end{eqnarray*}
a contradiction.
Therefore, $x \not \in i\Bbb{R} y$.

Conversely, let
$\hbox{Re}\langle \Lambda_k(x), \Lambda_k(y) \rangle=0$
for some nonzero elements
$x, y \in H$
with $x \not \in i\Bbb{R} y$.
Then
$\langle \Lambda_k(f), \Lambda_k(f) \rangle =\langle \Lambda_k(g), \Lambda_k(g) \rangle,$
where
 $f=x+y$  and $g=x-y$.
Suppose that
$f= \alpha g$ with $|\alpha|=1$.
Then
 $\|f\|=\|g\|$ and $x=\beta y$ with $\beta\not =0$.
Thus
$|\beta+1|=|\beta-1|,$
\begin{eqnarray*}
f=(\beta+1)y\quad\hbox{and}\quad g=(\beta-1)y.
\end{eqnarray*}
If $H$ is a real Hilbert space,
then $\beta$
 is real and hence $\beta=0$.
This is a contradiction.
If $H$ is a complex Hilbert space, then $\beta \in \Bbb{C}$.
Thus
$$|\beta+1|=|\beta-1|$$
holds only if
 $\beta=ia$
 for some $a \in \Bbb{R}$,
which implies that $x \in i \Bbb{R} y$.
These contradictions lead to $\{\Lambda_k\}_{k \in \Bbb{Z}}$ is not phase retrieval.
\end{proof}
\begin{corollary}\label{j2-123}
Let $\{ \Lambda_k \}_{k \in \Bbb{Z}}$ be a sequence in
${\cal L}(H)$ and $\hbox{dim}\; H>1$.
If $\{ \Lambda_k \}_{k \in \Bbb{Z}}$
 have a nonzero common root,
then $\{ \Lambda_k \}_{k \in \Bbb{Z}}$ is not phase retrieval. 
\end{corollary}
\begin{proof}
Let $x$ be a nonzero element of $H$ such that $\Lambda_k(x)=0$ for all $k\in\Bbb{Z}$.
Since
$\hbox{dim}\; H>1$,
 there exists a nonzero element $y\in H$ with
 $y\not\in i\Bbb{R}x$.
The proof will be completed if we only note that $Re\langle\Lambda_k(x), \Lambda_k(y)\rangle=0$
 for all $k\in\Bbb{Z}$.
\end{proof}
\begin{corollary}\label{j2-8}
Every phase retrievable sequence in ${\cal L}(H)$
separates points of  $H$.
\end{corollary}
\begin{proof}
Let $\{\Lambda_k\}_{k \in \Bbb{Z}}$ be a phase retrievable  sequence.
If $x$ is a nonzero element of $H$,
then by Theorem \ref{j2-7} there exists $k_0\in\Bbb{Z}$
such that $\Lambda_{k_0}(x)\neq 0$.
This implies that
$\{\Lambda_k\}_{k \in \Bbb{Z}}$
separates points of  $H$.
\end{proof}
Let us recall that
$\ell_2^d(\Bbb{C})$
is the set of all
$(x_1, x_2, \ldots, x_d)$
such that
$x_i \in \Bbb{C}$
 and  $\sum_{i=1}^{d} |x_i|^2 < \infty.$
For a sequence
$\{ \Lambda_k \}_{k \in \Bbb{Z}}$
in
${\cal L}(\ell_2^d(\Bbb{C}))$ and $x\in \ell_2^d(\Bbb{C})$,
let
\begin{eqnarray*}
W_x:= \overline{span} \{ Re(\Lambda_k(x)) \oplus Im (\Lambda_k (x)): k \in \Bbb{Z} \}.
\end{eqnarray*}
For every $x, y\in \ell_2^d(\Bbb{C})$ we have
\begin{eqnarray*}
2Re\langle \Lambda_k (y), \Lambda_k(x) \rangle &=&\langle \Lambda_k (y), \Lambda_k(x) \rangle + \overline{\langle \Lambda_k (y), \Lambda_k(x) \rangle }\\
&=& \langle Re(\Lambda_k (y)) \oplus Im(\Lambda_k(y)),  Re(\Lambda_k (x)) \oplus Im(\Lambda_k(x)) \rangle.
\end{eqnarray*}
This implies that
\begin{eqnarray*}
Re\langle \Lambda_k (y), \Lambda_k(x) \rangle=0
\end{eqnarray*}
if and only if
\begin{eqnarray*}
Re(\Lambda_k (y)) \oplus Im(\Lambda_k(y))\in W_x^\perp
\end{eqnarray*}
for all $k\in\Bbb{Z}$;
 or equivalently,
\begin{eqnarray*}
W_y\subseteq W_x^\perp.
\end{eqnarray*}
 If $z=ix$,
then
$Re\langle \Lambda_k (z), \Lambda_k(x) \rangle=0$.
This shows that $W_{ix}$ is a one dimensional subspace of $W_x^\perp$.
So
$W_x^{\perp}$ is non-empty and
$dim W_x^{\perp} \geq 1$.
 Note that $dim W_x^{\perp}>1$ if and only if $W_x^\perp\setminus W_{ix}$ is non-empty.
\begin{theorem}\label{nescafe}
 Let $\{\Lambda_k\}_{k \in \Bbb{Z}}$ be a sequence in ${\mathcal L}(\ell_2^d(\Bbb{C}))$.
 Then $\{\Lambda_k\}_{k \in \Bbb{Z}}$
is not phase retrieval if and only if there exist nonzero elements
$x, y\in \ell_2^d(\Bbb{C})$
 such that $y\not\in i\Bbb{R}x$
and
 $W_y$ is contained in $W_x^\perp.$
\end{theorem}
\begin{proof}
Suppose that 
$\{\Lambda_k\}_{k \in \Bbb{Z}}$ is not phase retrieval. 
By Theorem \ref{j2-7},
there are nonzero elements
$x, y\in \ell^2_d(\Bbb{C})$
such that $y\not\in i\Bbb{R}x$
and
\begin{eqnarray*}
Re\langle \Lambda_k (y), \Lambda_k(x) \rangle=0.
\end{eqnarray*}
This shows that $W_y\subseteq W_x^\perp.
$ A similar argument proves the converse.
\end{proof}

As an immediate consequence of
 Corollary \ref{j2-123}
 we have the following result.
\begin{corollary}
Let $\{\Lambda_k\}_{k \in \Bbb{Z}}$
 be a sequence in ${\mathcal L}(\ell_2^d(\Bbb{C}))$
and $d>1$.
If there exists a nonzero element
$x\in \ell_2^d(\Bbb{C})$
such that $W_y=\{0\}$.
 Then
$\{\Lambda_k\}_{k \in \Bbb{Z}}$
is not phase retrieval.
\end{corollary}
\begin{proposition}\label{j2-9}
Let $\{\Lambda_k\}_{k \in \Bbb{Z}}$ be a sequence in ${\mathcal L}(\ell_2^d(\Bbb{C}))$.
If there exist nonzero elements
$x, y\in \ell_2^d(\Bbb{C})$
such that $W_y$ is contained in
$W_x^{\perp}\setminus W_{ix}$,
then $\{\Lambda_k\}_{k \in \Bbb{Z}}$
is not phase retrieval.
The converse holds if
 $\Lambda_{k_0}$
 is injective for some $k_0\in\Bbb{Z}$.
 In this case, $dim W_x^{\perp}>1$.
\end{proposition}
\begin{proof}
Let there exist $k_0\in\Bbb{Z}$ with $\Lambda_{k_0}$ be injective.
Suppose that $\{\Lambda_k\}_{k \in \Bbb{Z}}$
is not phase retrieval.
By Theorem \ref{nescafe}, there are nonzero elements $x, y\in \ell^2_d(\Bbb{C})$
such that
$y\not\in i\Bbb{R}x$
and
$W_y$ is contained in $W_x^\perp.$
Let
\begin{eqnarray*}
Re(\Lambda_{k_0} (y)) \oplus Im(\Lambda_{k_0} (y)) \in W_{ix}.
\end{eqnarray*}
This shows that
\begin{eqnarray*}
Re(\Lambda_{k_0} (y)) = -a Im(\Lambda_{k_0} (x)) ~~~ \text{and} ~~~ Im(\Lambda_{k_0} (y))= a Re(\Lambda_{k_0} (x))
\end{eqnarray*}
for some $a\in\Bbb{R}$.
So
\begin{eqnarray*}
\Lambda_{k_0}(y) &=& Re(\Lambda_{k_0} (y)) + i Im(\Lambda_{k_0} (y)) \\
&=& -a Im(\Lambda_{k_0} (x)) + ia Re(\Lambda_{k_0} (x))  \\
&=& a Re(\Lambda_{k_0} (ix)) + a Im(\Lambda_{k_0} (ix))  \\
&=&\Lambda_{k_0} (iax).
\end{eqnarray*}
By injectivity of $\Lambda_{k_0}$,
we have $y=iax$, a contradiction.
Thus
$$W_y \subset W^{\perp}_x \setminus W_{ix}.$$
Therefore,
$W^{\perp}_x \setminus W_{ix}$
is non-empty which implies that
$dim{W_x^\perp}>1$.
The remain of proof follows from Theorem \ref{nescafe}.
\end{proof}

Let us recall that for every $x, y\in H$, the bounded operator $x\otimes y$ on $H$ is defined by
 $(x \otimes  y)(z) = \langle z, y \rangle x$
for all $z \in H$.
\begin{proposition}\label{j2-11}
 Let $\{ \Lambda_k \}_{k=1}^n$ be a sequence in
${\cal L}(H)$.
If  $\{\Lambda_k \}_{k =1}^n$
 is phase retrieval,  then
$\sum_{k =1}^n(\Lambda_k^*\Lambda_k(x) \otimes \Lambda_k^*\Lambda_k(x))$
is injective  for all nonzero vector $x \in H$.
\end{proposition}
\begin{proof}
Suppose that
$\sum_{k=1}^n (\Lambda_k^*\Lambda_k(x) \otimes \Lambda_k^*\Lambda_k(x))$
is not injective for some
$x \not =0$.
Then there exists $y \not =0$ such that
\begin{eqnarray*}
\langle \Lambda_k(y), \Lambda_k(x) \rangle=0.
\end{eqnarray*}
It follows that
\begin{eqnarray*}
Re\langle \Lambda_k(y), \Lambda_k(x) \rangle=0
\end{eqnarray*}
for all $k=1, \ldots, n$.
If $y = iax$  for some  $a \in  \Bbb{R}$,
then $\langle \Lambda_k(y), \Lambda_k(x) \rangle=0$,
which shows that $\Lambda_k(x) =0$.
This contradicts Corollary \ref{j2-8}.
So $y \not \in  i \Bbb{R} x$.
Therefore, $\{\Lambda_k \}_{k \in \Bbb{Z}}$ is
not phase retrieval, a contradiction.
\end{proof}
\begin{theorem}
Let $\{\Lambda_k\}_{k \in \Bbb{Z}}$ be a sequence in $\mathcal{L}(H)$ and let $T\in \mathcal{L}(H)$ be injective.
Then  $\{\Lambda_k\}_{k \in \Bbb{Z}}$ is phase retrieval if and only if $\{\Lambda_k T\}_{k \in \Bbb{Z}}$ is  phase retrieval.
\end{theorem}
\begin{proof}
Suppose that there exist nonzero elements $x, y\in H$ such that $y\not\in i\Bbb{R}x$ and $Re\langle \Lambda_k T( x), \Lambda_k T(y)\rangle=0$. Since $T$ is injective, $Ty\not\in i\Bbb{R}Tx$. Thus if $\{\Lambda_k\}_{k \in \Bbb{Z}}$ is phase retrieval, then $\{\Lambda_k T\}_{k \in \Bbb{Z}}$ is  phase retrieval.

Conversely, if for every $x,y\in H$ and $k\in\Bbb{Z}$,
$\|\Lambda_k (x)\|=\|\Lambda_k (y) \|$,
then there exists
 $\tilde{T}\in \mathcal{L}(H)$
such that
 $\|\Lambda_kT ( \tilde{T}x)\|=\|\Lambda_k T ( \tilde{T}y) \|$. Since
$\{\Lambda_k T\}_{k \in \Bbb{Z}}$ is phase retrieval
and $T$ is injective,
 $x = \alpha y$ with $|\alpha|=1$. Therefore, $\{\Lambda_k\}_{k \in \Bbb{Z}}$ is phase retrieval.
\end{proof}
In the following, let $H$  denote a separable infinite dimensional Hilbert space and  $\{H_k\}_{k \in \Bbb{Z}}$  be a sequence of subspaces of $H$.
Let $ \Lambda_k$ be a linear operator from $H$ into $H_k$ for all $k \in \Bbb{Z}$. Then the sequence  $\{\Lambda _k \} _{k \in \Bbb{Z}}$ is said to be a
\emph{$g$-frame for $H$ with respect to}
$\{H_k\}_{k \in \Bbb{Z}}$ (briefly, $g$-frame)
 if there exist two constants  $0 < A \leq  B$ such that
\begin{eqnarray*}
A ~\|  x\|^ 2 \leq  \sum_{k \in \Bbb{Z}} \| \Lambda _k (x) \|^2 \leq B ~ \|x\|^2.
\end{eqnarray*}
\begin{proposition}\label{j2-2}
Let $\{\Lambda_k\}_{k \in \Bbb{Z}}$ be a sequence in $\mathcal{L}(H)$ and let $T\in \mathcal{L}(H)$ be surjective.
Then  $\{\Lambda_k\}_{k \in \Bbb{Z}}$ is a phase retrievable $g$-frame if and only if $\{\Lambda_k T\}_{k \in \Bbb{Z}}$ is  a phase retrievable $g$-frame.
\end{proposition}
\begin{proof}
Let  $\{\Lambda_k \}_{k \in \Bbb{Z}}$ be a $g$-frame for $H$.
Then there exists $A, B>0$ such that for  every $\in H$,
\begin{eqnarray*}
A \|x \|^2 \leq \sum_{k \in \Bbb{Z}} \|\Lambda_k(x)\|^2 \leq B \|x\|^2.
\end{eqnarray*}
If $T$ is surjective,
 then $\|x\|=\|\overline{T} T x \| \leq \|\overline{T}\|  \| T x \|$ for some $\overline{T}\in \mathcal{L}(H)$. So
\begin{eqnarray*}
 \sum_{k \in \Bbb{Z}} \|\Lambda_k T (x)\|^2 &=&  \sum_{k \in \Bbb{Z}} \|\Lambda_k (T x)\|^2 \\
&\leq& B \|T x \|^2 \\
&\leq& B \|T\|^2 \|x\|^2
\end{eqnarray*}
and
\begin{eqnarray*}
 \sum_{k \in \Bbb{Z}} \|\Lambda_k T (x)\|^2  &\geq& A \|Tx \|^2 \geq A \|\overline{T}\|^{-2} \|x\|^2
\end{eqnarray*}
for all $x\in H$. Hence
$\{\Lambda_k T\}_{k \in \Bbb{Z}}$ is a $g$-frame. Now, let for every $x, y\in H$ and $k\in\Bbb{Z}$,
$\|\Lambda_k(Tx)\|=\|\Lambda_k(Ty)\|$.
Since
$\{\Lambda_k\}_{k \in \Bbb{Z}}$ is  phase retrieval,
there exists $\alpha\in \Bbb{C}$ such that
$|\alpha|=1$
and
$Tx= \alpha Ty$.
Surjectivity of $T$ implies that $x=\alpha y$.
Therefore, $\{\Lambda_k T\}_{k \in \Bbb{Z}}$ is a phase retrievable $g$-frame.

To prove the converse, let $z\in H$. Then $Tz^\prime=z$ for some $z^\prime\in H$. If $\{\Lambda_k T\}_{k \in \Bbb{Z}}$ is a $g$-frame, then there exist $A, B>0$ such that
\begin{eqnarray*}
A\|z\|^2&=&A\|Tz^\prime\|^2\\
&\leq&\sum_{k \in \Bbb{Z}}\|\Lambda_kT(z^\prime)\|^2\\
&=&\sum_{k \in \Bbb{Z}}\|\Lambda_k(z)\|^2\\
&\leq& B\|z\|^2.
\end{eqnarray*}
This shows that $\{\Lambda_k\}_{k \in \Bbb{Z}}$ is a $g$-frame.  Suppose that $\{\Lambda_k\}_{k \in \Bbb{Z}}$
 is not phase retrieval. Thus there exists nonzero elements $x, y\in H$ such that $y\not\in i\Bbb{R}x$ and
\begin{eqnarray*}
Re\langle\Lambda_k(x), \Lambda_k(y)\rangle=0.
\end{eqnarray*}
Choose $x^\prime, y^\prime\in H$ with
\begin{eqnarray*}
Tx^\prime=x\quad\hbox{and}\quad Ty^\prime=y.
\end{eqnarray*}
Note that
\begin{eqnarray*}
Re\langle\Lambda_kT(x^\prime), \Lambda_kT(y^\prime)\rangle=0
\end{eqnarray*}
and $Ty^\prime\not\in Tx^\prime$. Hence $\{\Lambda_kT\}_{k \in \Bbb{Z}}$ is not phase retrieval.
\end{proof}
Let   $\{\Lambda_k \}_{k \in \Bbb{Z}}$ be a  $g$-frame for $H$. The operator  $S: H \rightarrow  H$  defined by
\begin{eqnarray} \nonumber
S(x) := \sum_{k \in \Bbb{Z}} {\Lambda}_{k}  ^{*}  \, {\Lambda}_{k} \, (x)   \nonumber
\end{eqnarray}
is called the  $\textit{frame operator}$ of $\{\Lambda _k \} _{k \in \Bbb{Z}}$.
It has been proved that  the operator $S$  is bounded,  positive and invertible.
For $g$-frame $\{\Lambda_k \}_{k \in \Bbb{Z}}$, the sequence  $\{\Lambda_k S^{-1}\}_{k \in \Bbb{Z}}$  is called the $\textit{canonical dual}$ of $\{\Lambda_k \}_{k \in \Bbb{Z}}$; see [17, 18] for more details.

\begin{corollary}\label{j2-4}
Let   $\{\Lambda_k \}_{k \in \Bbb{Z}}$ be a  $g$-frame for $H$.
 Then  the canonical dual of $\{\Lambda_k\}_{k \in \Bbb{Z}}$
 is  phase retrieval  if and only if  $\{\Lambda_k\}_{k \in \Bbb{Z}}$ is  phase retrieval.
\end{corollary}
\section{Phase retrievable frames}
 A sequence $\{x_k\}_{k \in \Bbb{Z}}$  of elements of a separable infinite dimensional Hilbert space $H$ is  called  \emph{frame} if there exist constants
$0 < A \leq B$ such that
\begin{eqnarray} \nonumber
A \, \| x \|^2 \leq \sum_{k \in \Bbb{Z}} |\langle  x,x_k \rangle|^2 \leq  B\, \|x \|^2
\end{eqnarray}
 for all $x\in H$. For extensive of frames see [9]; see also [10] for $K$-frames, see [8] for fusion frames and [13, 14, 17] for $K$-$g$-frames and $g$-frames. Let the linear operator
$\Lambda_k: H\rightarrow H$ defined by
\begin{eqnarray*}
\Lambda_k(x)=\langle x, x_k\rangle x_k.
\end{eqnarray*}
One can easy prove that $\{\Lambda_k\}_{k \in \Bbb{Z}}$ is a phase retrievable $g$-frame  if and only if
$\{x_k\}_{k \in \Bbb{Z}}$ is a  phase retrievable frame.
\begin{theorem} \label{j1}
Let $T \in \mathcal{L}(H)$ be an injective operator.
If $\{x_k\}_{k \in \Bbb{Z}}$ is a phase retrievable frame for $H$, then
$\{Tx_k\}_{k \in \Bbb{Z}}$ is a phase retrievable frame.
The converse holds when $T$ is self-adjoint.
\end{theorem}
\begin{proof}
Let $T$ be injective. Then there exists $\overline{T} \in \mathcal{L}(H)$ such that
\begin{eqnarray*}
T \overline{T}= I=\overline{T}^* T^*
\end{eqnarray*}
on $H$.
So for every $x \in H$ we have
\begin{eqnarray*}
\|x\| = \| T  \overline{T} x \| \leq \|T\| \|  \overline{T} x \| \quad\hbox{and}\quad\|x\| = \| \overline{T}^* T^* x\| \leq \|\overline{T}^* \| \|T^* x\|.
\end{eqnarray*}
It follows that
\begin{eqnarray} \label{j90}
 \|T \|^{-1} \|x\| \leq \| \overline{T} x \|
\end{eqnarray}
and $\|\overline{T}^* \|^{-1} \|x\| \leq \|T^*x\|$.
Assume that $\{x_k\}_{k \in \Bbb{Z}}$ is a frame for $H$.
Then there exist $A, B >0$ such that for every $x \in H$
\begin{eqnarray}
\sum_{k \in \Bbb{Z}} |\langle x, Tx_k \rangle |^2  &=&
\sum_{k \in \Bbb{Z}} |\langle T^*x, x_k \rangle |^2 \nonumber \\
&\geq& A \|T^* x \|^2 \nonumber \\
&\geq& A  \|\overline{T}^* \|^{-2} \|x\|^2  \nonumber
\end{eqnarray}
and
\begin{eqnarray*}
\sum_{k \in \Bbb{Z}} |\langle x, Tx_k \rangle |^2  =
\sum_{k \in \Bbb{Z}} |\langle T^*x, x_k \rangle |^2 \leq B \|T^*\|^2 \|x\|^2.
\end{eqnarray*}
Hence $\{Tx_k\}_{k \in \Bbb{Z}}$  is a frame for $H$.
Let  
$|\langle x, Tx_k \rangle |= |\langle y, Tx_k \rangle |$ for all $x, y\in H$ and $k \in \Bbb{Z}$. Then 
$|\langle T^*x, x_k \rangle |= |\langle T^*y, x_k \rangle |$. 
Now, since $\{x_k\}_{k \in \Bbb{Z}}$ is phase-retrieval, there exists $\alpha\in\Bbb{C}$ with $|\alpha|=1$ such that
\begin{eqnarray*} 
T^* x = \alpha T^* y. 
\end{eqnarray*}
Then $\overline{T}^* T^* x=  \alpha \overline{T}^* T^* y$. Therefore, $x=\alpha y$. 
That is, $\{Tx_k\}_{k \in \Bbb{Z}}$  is a  phase-retrieval frame. 

Conversely, let $\{Tx_k \}_{k \in \Bbb{Z}}$ be a frame.
Then there exists $A, B >0$ such that for every $x\in H$
\begin{eqnarray*}
\sum_{k \in \Bbb{Z}} |\langle x, x_k \rangle |^2 &=&
\sum_{k \in \Bbb{Z}} |\langle  T  \overline{T} x, x_k \rangle |^2 \nonumber \\
&=& \sum_{k \in \Bbb{Z}} |\langle  \overline{T} x, T^* x_k \rangle |^2 \nonumber \\
&=& \sum_{k \in \Bbb{Z}} |\langle  \overline{T} x, T x_k \rangle |^2. \nonumber
\end{eqnarray*}
From (\ref{j90}) we obtain
\begin{eqnarray*}
A \|T\|^{-2} \|x\|^2 & \leq& A \|\overline{T} x \|^2  \nonumber \\
&\leq& \sum_{k \in \Bbb{Z}} |\langle x, x_k \rangle |^2 =\sum_{k \in \Bbb{Z}} |\langle  \overline{T} x, T x_k \rangle |^2 \nonumber \\
&\leq& B \|\overline{T}\|^2 \|x\|^2.
\end{eqnarray*}
So, $\{x_k \}_{k \in \Bbb{Z}}$ is a  frame for $H$.
Now, let
$|\langle x, x_k \rangle | = |\langle y, x_k \rangle | $ for all $k \in \Bbb{Z}$ and $x,y \in H$.
Since $T$ is injective and self-adjoint, it follows that
\begin{eqnarray*}
|\langle \overline{T}x, Tx_k \rangle |= |\langle \overline{T}y, Tx_k \rangle |.
\end{eqnarray*}
If $\{Tx_k \}_{k \in \Bbb{Z}}$ is  phase retrieval, then there exists $\alpha\in\Bbb{C}$ with $|\alpha|=1$ and
\begin{eqnarray*}
 \overline{T}x = \alpha \overline{T} y.
\end{eqnarray*}
Consequently,
 $T \overline{T}x = \alpha  T \overline{T} y$.
Thus $\{x_k\}_{k \in \Bbb{Z}}$ is a phase retrievable frame.
\end{proof}
As a consequence of Theorem \ref{j1} we give the following result.
\begin{corollary} \label{j3}
Let  $\{x_k\}_{k \in \Bbb{Z}}$ be a frame for $H$.
The canonical dual of  $\{x_k\}_{k \in \Bbb{Z}}$  is
phase retrieval  if and only if $\{x_k\}_{k \in \Bbb{Z}}$ is phase retrieval. 
\end{corollary}
\begin{proposition} \label{j4}
Let $\{x_k\}_{k \in \Bbb{Z}}$ be a phase retrievable frame for $H$ with dual frame $\{y_k\}_{k \in \Bbb{Z}}$.
If $\{y_k\}_{k \in \Bbb{Z}}$ is an orthogonal sequence,
then  $\{y_k\}_{k \in \Bbb{Z}}$ is a phase retrievable frame.
\end{proposition}
\begin{proof}
Let
 $|\langle x, y_k \rangle | = |\langle y, y_k \rangle |$ for all $x,y \in H$ and $k \in \Bbb{Z}$.
Then
\begin{eqnarray*}
\langle x,y_k \rangle \overline{\langle x,y_k \rangle}=
\langle y,y_k \rangle \overline{\langle y,y_k \rangle},
\end{eqnarray*}
i.e.,
\begin{eqnarray*}
\langle x, y_k \rangle \langle y_k,x\rangle=
\langle y,y_k \rangle \langle y_k,y \rangle.
\end{eqnarray*}
Put $\alpha_k=\langle y_k,x \rangle$ and
$\beta_k= \langle y_k,y \rangle$. Then
\begin{eqnarray*}
\langle x, \overline{\alpha_k}  \, y_k \rangle =
\langle y, \overline{\beta_k}  \, y_k \rangle.
\end{eqnarray*}
Since $\{x_k\}_{k \in \Bbb{Z}}$ and $\{y_k\}_{k \in \Bbb{Z}}$
are dual frame, we have
\begin{eqnarray*}
x= \sum_{j \in \Bbb{Z}} \langle x, x_j \rangle y_j\,\,\,\,\, \hbox{and} \,\,\,\,\,\,    y= \sum_{j \in \Bbb{Z}} \langle y, x_j \rangle y_j.
\end{eqnarray*}
It follows that
\begin{eqnarray*}
\bigg\langle \sum_{j \in \Bbb{Z}} \langle x, x_j \rangle y_j, \overline{\alpha_k}  y_k \bigg\rangle = \bigg\langle  \sum_{j \in \Bbb{Z}} \langle y, x_j \rangle y_j, \overline{\beta_k} y_k \bigg\rangle.
\end{eqnarray*}
This together with the orthogonality of $\{y_k\}_{k \in \Bbb{Z}}$ shows that  for every $k\in\Bbb{Z}$
\begin{eqnarray*}
  \langle x, x_k \rangle \,  \alpha_k  \, \langle y_k, y_k \rangle = \langle y, x_k \rangle \,  \beta_k \,  \langle y_k, y_k \rangle.
\end{eqnarray*}
Thus
\begin{eqnarray*}
 \sum_{k \in \Bbb{Z}}  \langle x, x_k \rangle  \langle y_k, x \rangle =   \sum_{k \in \Bbb{Z}}  \langle y, x_k \rangle \langle y_k, y \rangle.
\end{eqnarray*}
From Lemma $6.3.2$ in [9] we infer that $x=y$.
Therefore, $\{y_k\}_{k \in \Bbb{Z}}$ is a phase retrievable frame.
\end{proof}
\begin{proposition} \label{j7}
Let  $T \in \mathcal{L}(H)$ be  an injective  self-adjoint operator and $\lambda>0$.
 Let also $\{x_k\}_{k \in \Bbb{Z}}$ and $\{y_k\}_{k \in \Bbb{Z}}$ be sequences in $H$ such that
for every $k\in\Bbb{Z}$
\begin{eqnarray}\label{f7}
|\langle x, x_k  \pm y_k \rangle | = \lambda | \langle x,  x_k \rangle |.
\end{eqnarray}
If either
$\{x_k\}_{k \in \Bbb{Z}}$ or $\{Tx_k\}_{k \in \Bbb{Z}}$
 is a phase retrievable frame,
then $\{x_k \pm y_k\}_{k \in \Bbb{Z}}$ is  a phase retrievable frame.
\end{proposition}
\begin{proof}
Let  $\{x_k\}_{k \in \Bbb{Z}}$ be a frame for $H$ with bounds $A$ and $B$.
It is easy to see that
 $\{x_k \pm y_k\}_{k \in \Bbb{Z}}$
 is a frame for $H$
with bounds $\lambda A$ and $\lambda B$.
 To complete the proof, let $x, y\in H$ such that
$|\langle x, x_k \pm y_k \rangle | = |\langle y, x_k \pm y_k \rangle | $
 for all $k\in \Bbb{Z}$. Then by (\ref{f7})
\begin{eqnarray*}
\lambda |\langle x, x_k \rangle | =  \lambda|\langle y, x_k \rangle |.
\end{eqnarray*}
So if $\{x_k\}_{k \in \Bbb{Z}}$ is a phase retrievable frame for $H$,
then $\{x_k \pm y_k\}_{k \in \Bbb{Z}}$ is a phase retrievable frame for $H$.
 In the case where,
$\{Tx_k\}_{k \in \Bbb{Z}}$  is a phase retrievable frame,
the result is proved if we only note that by Theorem \ref{j1},  $\{x_k \}_{k \in \Bbb{Z}}$ is a phase retrievable frame for $H$.
\end{proof}

As an immediate consequence of Theorem \ref{j1} and Proposition \ref{j7} we have the following result.
\begin{corollary} \label{j8}
Let $\{x_k\}_{k \in \Bbb{Z}}$ be a phase retrievable frame for $H$ and  $\{y_k\}_{k \in \Bbb{Z}}$ be a sequence in $H$.
Let $T \in \mathcal{L}(H)$
be injective.
If there exists $\lambda>0$ such that for every $k$
\begin{eqnarray*}
|\langle x, T (x_k  \pm y_k) \rangle | = \lambda | \langle x, T x_k \rangle |,
\end{eqnarray*}
then $\{T(x_k \pm  y_k)\}_{k \in \Bbb{Z}}$
is  a phase retrievable frame for $H$.
\end{corollary}

\section{Compliance with ethical standards}
\textbf{Conflict of interest:} All authors declare that they have no conflict of
interest.\\
\textbf{Data Availability Statement:} No data sets were generated or analysed during the current study.


\begin{thebibliography}{99}

\bibitem{Ala}
R.~Alaifari and P.~Grohs, \textit{Phase retrieval in the general setting of continuous frames for Banach spaces},  SIAM J. Math. Anal., \textbf{49}(3) (2017), 1895-1911. 

\bibitem{Bahman}
S.~Bahmanpour, J.~Cahill, P.~G.~Casazza, J.~Jasper and  L.~M. ~Woodland, \textit{Phase retrieval and norm retrieval}, J. Contemp. Math,  Amer. Math. Soc., \textbf{650} (2015), 3--14.

\bibitem{Balan}
R.~Balan, P. G.~Casazza and D.~Edidin,  \textit{On signal reconstruction without noisy phase}, Appl. Comput. Harmon.  Anal, \textbf{20} (2006), 345--356.

\bibitem{Bandeira}
A.~S.~Bandeira, Y.~Chen and  D.~G.~Mixon, \textit{Phase retrieval from power spectra of masked signals}, J. Inf.  inference., \textbf{3} (2014), 83--102.

\bibitem{Bodmann}
B.~G.~Bodmann and N.~Hammen,  \textit{Stable phase retrievable with low-redundancy frames}, J. Adv. Comput.
Math., \textbf{41} (2015), 317-33.

\bibitem{Bojarovska}
I.~Bojarovska and A.~Flinth, \textit{Phase retrieval from Gabor measurements}, J. Fourier Anal. Appl., \textbf{22} (2016), 542-567.

\bibitem{Cahill}
J.~Cahill, P.~Casazza, J.~Peterson and L.~Woodland, \textit{Phase retrieval by projections}, Houston J. Math.,
\textbf{42} (2) (2016), 537--558.

\bibitem{Casazza}
P.~Casazza and G.~Kutyniok, \textit{Frames of subspaces}, J. Math. Anal. Appl., \textbf{38}  (2) (2005), 541--553.

\bibitem{Ole}
O.~Christensen, \textit{An Introduction to Frames and Riesz Bases}, Birkhauser, Boston, 2003.

\bibitem{Gavruta}
L.~Gavruta,  \textit{Frames for operators},  Appl. Comput. Harmon. Anal., \textbf{32} (2012), 139--144.

\bibitem{Han1}
D.~Han and T.~Juste, \textit{Phase-retrievable operator-valued frames and representations of quantum channels},  J. Linear. Algebra. Appl., \textbf{579} (2019), 148--168.

\bibitem{Han2}
D.~Han,  T.~Juste, Y.~Li and W. Sun, \textit{Frame Phase-Retrievability and Exact Phase-Retrievable Frames},  J. Fourier Anal. Appl., \textbf{25} (2019), 3154--3173.

\bibitem{Jahedi} 
S.~Jahedi, F.~Javadi and M. J.~Mehdipour, \textit{On $g$-frame representations via linear operators}, J. Pseudo-Differ. Oper. Appl., {\textbf 14} (2023), 52.

\bibitem{Javadi}
F.~Javadi and M. J.~Mehdipour, \textit{On constructions of $K$-$g$-frames in Hilbert spaces}, Mediterr. J. Math., {\textbf 18} (2021), 210.

\bibitem{Razghandi}
A.~Razghandi and R.~Raisi Tousi, \textit{Using tensor product dual frames for phase retrieval problems},  J. Pseudo-Differ. Oper. Appl., {\textbf 12}  (2021), 26. 

\bibitem{Renes}
J.~M.~Renes, R.~Blume-Kohout, A.~J.~Scott and C.~M.~Caves, \textit{Symmetric informationally complete
quantum measurements}, J. Math. Phys., \textbf{45} (2004), 2171--2180.

\bibitem{Sun}
 W.~Sun, \textit{$G$-frames and $g$-Riesz bases},   J. Math. Anal. Appl., \textbf{322} (1) (2006), 437--452.

\bibitem{Sun1}
W.~Sun, \textit{Stability of $g$-frames},  J. Math. Anal. Appl., \textbf{326} (2) (2007), 858--868.
\end{thebibliography}
\end{document}